\documentclass[12pt]{amsart}
\usepackage[colorlinks=true, pdfstartview=FitV, linkcolor=blue, citecolor=blue, urlcolor=blue, breaklinks=true]{hyperref}
\usepackage{amsmath,amsfonts,amssymb,amsthm,amscd,comment,euscript}
\usepackage[usenames]{color}
\usepackage[all]{xy}
\usepackage{tikz}
\newcommand\C{\mathbb{C}}
\newcommand\Z{\mathbb{Z}}
\newcommand\Q{\mathbb{Q}}

\newcommand\N{\mathbb{N}}

\newcommand\id{\mathrm{id}}
\newcommand\rat{\mathrm{rat}}
\newcommand\ev{\mathrm{ev}}

\newcommand\g{\mathfrak{g}}

\newcommand\fm{\mathfrak{m}}

\newcommand\bx{\mathbf{x}}

\newcommand\bT{\mathbf{T}}
\newcommand\bU{\mathbf{U}}

\newcommand\cF{\mathcal{F}}

\newcommand\cS{\mathcal{S}}
\newcommand\cE{\mathcal{E}}

\newcommand{\ts}{\textstyle}

\DeclareMathOperator{\Hom}{Hom}
\DeclareMathOperator{\Ext}{Ext}

\DeclareMathOperator{\End}{End}
\DeclareMathOperator{\Aut}{Aut}
\DeclareMathOperator{\Int}{Int}
\DeclareMathOperator{\Out}{Out}

\DeclareMathOperator{\Spec}{Spec}
\DeclareMathOperator{\maxSpec}{maxSpec}

\DeclareMathOperator{\Supp}{Supp} % Support

\DeclareMathOperator{\wt}{wt}

\DeclareMathOperator{\height}{ht}
\DeclareMathOperator{\mult}{mult}

\theoremstyle{plain}
\newtheorem{theorem}{Theorem}[section]
\newtheorem{proposition}[theorem]{Proposition}
\newtheorem{lemma}[theorem]{Lemma}
\newtheorem{corollary}[theorem]{Corollary}

\theoremstyle{definition}
\newtheorem{definition}[theorem]{Definition}
\newtheorem*{remark*}{Remark}
\newtheorem{remark}[theorem]{Remark}

\newtheorem{example}[theorem]{Example}

\numberwithin{equation}{section}

\allowdisplaybreaks

\begin{document}
%
%%%%%%%%%%%%%%%%%%%%%%%%%%%%%%%%%%%

\title[Local Weyl modules for EMAs with free abelian group actions]{Local Weyl modules for equivariant map algebras \\ with free abelian group actions}

\author{Ghislain Fourier}
\address{G.~Fourier: Mathematisches Institut der Universit\"at zu K\"oln}
\email{gfourier@math.uni-koeln.de}
\author{Tanusree Khandai}
\address{T.~Khandai:}
\email{p.tanusree@gmail.com}
\author{Deniz Kus}
\address{D.~Kus: Mathematisches Institut der Universit\"at zu K\"oln}
\email{dkus@math.uni-koeln.de}
\author{Alistair Savage}
\address{A.~Savage: Department of Mathematics and Statistics, University of Ottawa}
\email{alistair.savage@uottawa.ca}
\thanks{The first and the third authors were partially sponsored by the DFG-Schwerpunktprogramm 1388 ``Darstellungstheorie''.  The research of the fourth author was supported by a Discovery Grant from the Natural Sciences and Engineering Research Council of Canada.}


%\date{\today}

\begin{abstract}
Suppose a finite group $\Gamma$ acts on a scheme $X$ and a finite-dimensional Lie algebra $\g$.  The associated \emph{equivariant map algebra} is the Lie algebra of equivariant regular maps from $X$ to $\g$.  Examples include generalized current algebras and (twisted) multiloop algebras.

Local Weyl modules play an important role in the theory of finite-dimensional representations of loop algebras and quantum affine algebras.  In the current paper, we extend the definition of local Weyl modules (previously defined only for generalized current algebras and twisted loop algebras) to the setting of equivariant map algebras where $\g$ is semisimple, $X$ is affine of finite type, and the group $\Gamma$ is abelian and acts freely on $X$.  We do so by defining \emph{twisting} and \emph{untwisting functors}, which are isomorphisms between certain categories of representations of equivariant map algebras and their untwisted analogues.  We also show that other properties of local Weyl modules (e.g.\ their characterization by homological properties and a tensor product property) extend to the more general setting considered in the current paper.
\end{abstract}

\maketitle \thispagestyle{empty}

\tableofcontents

%%%%%%%%%%%%%%%%%%%%%%%%%%%%%%%%%%%%%%%%%%%%%%%%%%%%%%%%%%%%%%%%%%%
%
\section*{Introduction}
%
%%%%%%%%%%%%%%%%%%%%%%%%%%%%%%%%%%%%%%%%%%%%%%%%%%%%%%%%%%%%%%%%%%%

Partially because of their importance in the theory of quantum affine Lie algebras, loop algebras $\g \otimes \C[t,t^{-1}]$, where $\g$ is a semisimple Lie algebra, have been the subject of intense study over the last two decades.  Their representation theory is particularly interesting because the category of finite-dimensional representations is not semisimple.  In \cite{C85,CP86}, it was shown that the irreducible objects in these categories are highest weight in a suitable sense, and a classification was given in terms of these highest weights, which are $n$-tuples of polynomials, where $n$ is the rank of $\g$.  In \cite{CP01}, it was shown that to each such $n$-tuple of polynomials $\pi$, there exists a unique largest highest weight module $W(\pi)$ of highest weight $\pi$. The modules $W(\pi)$, called \emph{(local) Weyl modules} by analogy with the modular representation theory of (the positive characteristic version of) $\g$, have the property that any finite-dimensional highest weight module of highest weight $\pi$ is a quotient of $W(\pi)$.

Weyl modules for loop algebras also play an important role in the representation theory of quantum affine algebras.  In particular, under a natural condition on their highest weight, the irreducible finite-dimensional representations of quantum affine algebras specialize at $q=1$ to representations of the loop algebras.  In this limit, the representations are no longer irreducible, but are quotients of the corresponding local Weyl module.  It was conjectured (and proved for $\g = \mathfrak{sl}_2$) in \cite{CP01} that all local Weyl modules are obtained as $q=1$ limits of irreducible finite-dimensional modules of quantum affine algebras.  In particular, this conjecture implies that the local Weyl modules are the classical limits of the standard modules defined by Nakajima in \cite{N01} and further studied by Varagnolo and Vasserot in \cite{VV02}.

In \cite{CP01}, Chari and Pressley defined the \emph{global Weyl modules} associated to dominant integral weights of $\g$.  These are the largest integrable highest weight modules of the given highest weight and were conjectured to be free modules for a certain commutative algebra.  This motivated a series of papers \cite{BN04,CL06,CM04,FoL07,N01,Na10} on local Weyl modules which computed their dimension and character, identified them with tensor products of Demazure modules, and eventually lead to the proof of this conjecture as well as the aforementioned conjecture that all local Weyl modules are $q=1$ limits of irreducible finite-dimensional modules of quantum affine algebras (for an arbitrary simple $\g$).

In \cite{FL04}, Feigin and Loktev extended the notion of global Weyl modules to the setting of generalized current algebras $\g\otimes A$, where $A$ is a commutative associative unital algebra over the complex numbers.  In the case that $A$ is the coordinate ring of an affine variety, they also extended the definition of local Weyl modules and obtained analogues of some of the results of \cite{CP01}.  In particular, they proved that these modules are finite-dimensional and that every local Weyl module is the tensor product of local Weyl modules associated to a single point (a property which is also true for finite-dimensional irreducible modules).

Motivated by the methods used to study the BGG-category $\mathcal O$ for semisimple Lie algebras, a functorial approach to the study of the Weyl modules for generalized current algebras was adopted in \cite{CFK}.  There it was shown that, via homological properties, one can naturally define more general Weyl modules for the Lie algebra $\g\otimes A$, where $A$ is a commutative associative unital algebra over the complex numbers.  This is done by defining the \emph{Weyl functor} from a suitable category of modules for a commutative algebra $\mathbf{A}_{\lambda}$ (these modules play the role of highest weight spaces) to the category of integrable modules for $\g \otimes A$ with weights bounded by a dominant integrable weight $\lambda$ of $\g$.  Under the condition that $A$ is finitely generated, it was shown that every local Weyl module is finite-dimensional.  Furthermore, the translation of the universal property of the Weyl module into the language of homological algebra yielded a simplified proof of the tensor product property.

The algebras mentioned above all are ``untwisted''.  There are natural twisted versions of loop algebras, related to the twisted affine Lie algebras.  More precisely, the twisted loop algebras are fixed point subalgebras of untwisted loop algebras $\g \otimes \C[t,t^{-1}]$ under the action of certain finite-order automorphisms.  Extending the ideas of \cite{CP01}, local Weyl modules for the twisted loop algebras were defined and studied in \cite{CFS08}, where it was realized that they can be identified with suitably chosen local Weyl modules for untwisted loop algebras.  It is thus natural to ask if twisted versions of local Weyl modules exist when one moves from loop algebras to the more general setting of generalized current algebras.

The twisted analogues of generalized current algebras are \emph{equivariant map algebras}.  Suppose $X = \Spec A$ is an affine scheme and $\g$ is a finite-dimensional Lie algebra, both defined over an algebraically closed field of characteristic zero, and that $\Gamma$ is a finite group acting on both $X$ (equivalently, on $A$) and $\g$ by automorphisms.  Then the equivariant map algebra $(\g \otimes A)^\Gamma$ is the Lie algebra of equivariant algebraic maps from $X$ to $\g$.  In the current paper, we will assume that $\g$ is semisimple, $X$ is of finite type, $\Gamma$ is abelian, and $\Gamma$ acts freely on $X$.  Even with these restrictions, equivariant map algebras are a large class of Lie algebras that include the above mentioned examples of (twisted) loop algebras and generalized current algebras as well as many others.

A complete classification of the irreducible finite-dimensional representations of an equivariant map algebra was given in \cite{NSS}.  Let $X_*$ denote the set of finite subsets of $X_\rat$, the set of rational points of $X$, that do not contain two points in the same $\Gamma$-orbit.  For $\bx \in X_*$, we have a surjective \emph{evaluation map}
\[ \textstyle
  \ev^\Gamma_\bx : (\g \otimes A)^\Gamma \to \g^\bx = \bigoplus_{x \in \bx} \g.
\]
An \emph{evaluation representation} is a representation of the form $\rho \circ \ev^\Gamma_\bx$, where $\rho = \bigotimes_{x \in \bx} \rho_x$ for representations $\rho_x : \g \to \End V_x$, $x \in \bx$.  In the setup of the current paper, the classification of \cite{NSS} says that all irreducible finite-dimensional representations are evaluation representations.  We define the \emph{support} of an irreducible finite-dimensional representation to be $\bigcup (\Gamma \cdot x)$, where the union is over the $x \in \bx$ such that $\rho_x$ is nontrivial.  For an arbitrary finite-dimensional representation, we define its support to be the union of the supports of its irreducible constituents.  This support depends only on the isomorphism class of the representation.

For an equivariant map algebra, one is not assured of the existence of a semisimple fixed point subalgebra $\g^\Gamma$ or a Cartan subalgebra of $(\g \otimes A)^\Gamma$ in the classical sense.  Since past approaches to the study of Weyl modules for twisted loop algebras rely heavily on the representation theory of $\g^\Gamma$, this is a major obstacle to generalizing such techniques to the more general setting of equivariant map algebras.  Furthermore, owing to the unavailability of the classical notion of weights for $(\g \otimes A)^\Gamma$-modules, the notion of highest weight  modules is not clear in this context. For these reasons, new techniques are needed.

Let $\cF$ and $\cF^\Gamma$ denote the category of finite-dimensional $(\g \otimes A)$-modules and $(\g \otimes A)^\Gamma$-modules respectively.  For $\bx \in X_*$, let $\cF_\bx$ (respectively $\cF^\Gamma_\bx$) denote the full subcategory of $\cF$ (respectively $\cF^\Gamma$) consisting of modules with support contained in $\bx$ (respectively $\Gamma \cdot \bx$).  Motivated by \cite{CFS08,Lau10,NSS} we define, for each $\bx \in X_*$, mutually inverse isomorphisms of categories
\[
  \xymatrix{
    \cF_\bx \ar@<1ex>[rr]^{\bT_\bx} & & \cF^\Gamma_\bx \ar@<1ex>[ll]^{\bU_\bx}
  }
\]
called \emph{twisting} and \emph{untwisting functors} (see Theorem~\ref{theo:category-equivalence}).  These functors allow us to move back and forth at will between the theory of finite-dimensional representations of equivariant map algebras (satisfying the assumptions of the current paper) and the corresponding theory for generalized current algebras.  In particular, to any irreducible finite-dimensional $(\g \otimes A)^\Gamma$-module $V$, we can associate a twisted local Weyl module as follows.  Let $\bx \in X_*$ contain one point in each $\Gamma$-orbit in the support of $V$.  Then $\bU_\bx V$ is an irreducible finite-dimensional $(\g \otimes A)$-module, to which is associated an (untwisted) local Weyl module $W$.  We then define the local Weyl module associated to $V$ to be $\bT_\bx W$, and one can show that this definition is independent of the choice of $\bx$ (see Proposition~\ref{prop:Weyl-restriction-independence}).

Apart from their role in the definition of the twisted local Weyl modules, the twisting and untwisting functors also allow us to use the characterization of local Weyl modules by homological properties given in \cite{CFK} to give a similar characterization of twisted local Weyl modules.  However, some subtlety is involved here.  The homological characterization given in \cite{CFK} involves certain categories of highest weight modules.  Since the Cartan subalgebra of $\g$ is not necessarily preserved by the action of the group $\Gamma$, such methods do not immediately carry over to the twisted setting.  In order to circumvent this problem, we replace the usual order on weights by another partial order arising from a suitably defined \emph{height} function on the weight lattice.  Our modified homological characterization is equivalent to the one given in \cite{CFK}, but has the advantage that it carries over to the twisted versions.

There are several natural questions arising from our treatment of local Weyl modules for equivariant map algebras.  For instance, can one define global Weyl modules (see \cite{CP01, CFK}) and is there an analogue of the algebra $\mathbf{A}_\lambda$ defined in \cite{CFK}?  Can one extend the results of the current paper to the case where the group $\Gamma$ does not act freely on $X$?  It would also be interesting to further examine the relationship between the twisting and untwisting functors defined here and connections between the representation theory of twisted and untwisted quantum affine algebras appearing in the literature (see, for example, \cite{Her10}).

The paper is organized as follows.  In Section~\ref{sec:ema-background} we recall the definition of equivariant map algebras and certain results on their finite-dimensional irreducible representations.  We introduce the twisting and untwisting functors in Section~\ref{sec:cat-equivalence} and prove that they are isomorphisms of categories.  In Section~\ref{sec:local-weyl-modules} we recall the results on local Weyl modules for generalized current algebras and then introduce the notion of local Weyl modules for equivariant map algebras.  We also show there that they satisfy a natural tensor product property.  Finally, in Section~\ref{sec:homological-properties} we give a characterization of the local Weyl modules by homological properties.

\subsection*{Acknowledgements}

The authors would like to thank E.~Neher for useful discussions.  The first, third, and fourth authors would also like to thank the Hausdorff Research Institute for Mathematics and the organizers of the Trimester Program on the Interaction of Representation Theory with Geometry and Combinatorics, during which the ideas in the current paper were developed.  The fourth author would like to thank the Institut de Math\'ematiques de Jussieu and the D\'epartement de Math\'ematiques d'Orsay for their hospitality during his stays there, when some of the writing of the current paper took place.

%%%%%%%%%%%%%%%%%%%%%%%%%%%%%%%%%%%%%%%%%%%%%%%%%%%%%%%%%%%%%%%%%%%
%
\section{Equivariant map algebras and their irreducible representations}
\label{sec:ema-background}
%
%%%%%%%%%%%%%%%%%%%%%%%%%%%%%%%%%%%%%%%%%%%%%%%%%%%%%%%%%%%%%%%%%%%

In this section, we review the definition of equivariant map algebras and the classification of their irreducible finite-dimensional representations given in \cite{NSS}.  Let $k$ be an algebraically closed field of characteristic zero and $A$ be unital associative commutative finitely generated $k$-algebra.  We let $X = \Spec A$, the prime spectrum of $A$ (so $X$ is an affine scheme of finite type).  A point $x \in X$ is called a \emph{rational point} if $A/\mathfrak{m}_x \cong k$, where $\mathfrak{m}_x$ is the ideal corresponding to $x$.  We denote the subset of rational points of $X$ by $X_\rat$.  Since $A$ is finitely generated, we have $X_\rat = \maxSpec A$. Suppose $\Gamma$ is a finite abelian group acting on $X$ (equivalently, on $A$) and on a semisimple Lie algebra $\g$ by automorphisms.  Let $\g \otimes A$ be the Lie $k$-algebra of regular maps from $X$ to $\g$.  This is a Lie algebra under pointwise multiplication.  The \emph{equivariant map algebra} $(\g \otimes A)^\Gamma$ consists of the $\Gamma$-fixed points of the canonical (diagonal) action of $\Gamma$ on $\g \otimes A$.  Thus $(\g \otimes A)^\Gamma$ is the subalgebra of $\Gamma$-equivariant maps.  In the current paper, we are interested in the case that $\Gamma$ acts freely on $X$, by which we mean that it acts freely on $X_\rat$.  We shall assume this is the case for the entirety of the paper.  Following the usual abuse of notation, we will use the terms `module' and `representation' interchangeably.

\begin{remark}
  We could consider the more general case where $\g$ is finite-dimensional reductive.  However, then $(\g \otimes A)^\Gamma \cong ([\g,\g] \otimes A)^\Gamma \oplus (Z(\g) \otimes A)^\Gamma$ as Lie algebras, \cite[(3.4)]{NS11}, where $[\g,\g]$ is semisimple and $Z(\g)$ is the centre of $\g$ (and so $(Z(\g) \otimes A)^\Gamma$ is an abelian Lie algebra).  The representation theory of $(\g \otimes A)^\Gamma$ thus essentially ``splits'' and so it suffices to consider the case of $\g$ semisimple.  See \cite{NS11} for details.
\end{remark}

We denote by $X_*$ the set of finite subsets $\bx\subseteq X_\rat$ for which $\Gamma \cdot x \cap \Gamma \cdot x' = \varnothing$ for distinct $x,x'\in \bx$. For $\bx\in X_*$, we define $\g^\bx = \bigoplus_{x\in \bx}\, \g$.  The evaluation map
\[
  \ev^\Gamma_\bx : (\g \otimes A)^\Gamma \to \g^\bx,\quad \ev^\Gamma_\bx (\alpha) =  (\alpha(x))_{x\in \bx},
\]
is a Lie algebra epimorphism \cite[Cor.~4.6]{NSS}.  To $\bx \in X_*$ and a set $\{\rho_x: x \in \bx\}$ of (nonzero) representations $\rho_x : \g \to \End_k V_x$, we associate the \emph{evaluation representation} $\ev^\Gamma_\bx(\rho_x)_{x\in \bx}$ of $(\g \otimes A)^\Gamma$, defined as the composition
\[ \ts
  (\g \otimes A)^\Gamma \xrightarrow{\ev^\Gamma_\bx} \g^\bx \xrightarrow{\bigotimes_{x \in \bx} \rho_x} \End_k \left( \bigotimes_{x \in \bx} V_x \right).
\]
If all $\rho_x, x\in \bx$, are irreducible finite-dimensional representations, then this is also an irreducible finite-dimensional representation of $(\g \otimes A)^\Gamma$, \cite[Prop.~4.9]{NSS}. The \emph{support}
of an evaluation representation $V= \bigotimes_{x \in \bx} V_x$, abbreviated $\Supp V$, is the union of all $\Gamma \cdot x$, $x \in \bx$, for which $\rho_x$ is not the one-dimensional trivial representation of $\g$.

Fix a triangular decomposition $\g = \mathfrak{n}^- \oplus \mathfrak{h} \oplus \mathfrak{n}^+$ and a set of simple roots for $\g$.  Let $P$ and $Q$ be the corresponding weight and root lattices respectively, and let $P^+$ denote the set of dominant integral weights.  For $\lambda \in P^+$, let $V(\lambda)$ be the corresponding irreducible representation of $\g$ of highest weight $\lambda$.  In this way we identify the set of isomorphism classes of irreducible finite-dimensional $\g$-modules with $P^+$.

It is well known that $\Aut \g \cong \Int \g \rtimes \Out \g$, where $\Int \g$ is the group of inner automorphisms of $\g$ and $\Out \g$ is the group of diagram automorphisms of $\g$.  The diagram automorphisms act naturally on $P$, $Q$, and $P^+$.  If $\rho$ is an irreducible representation of $\g$ of highest weight $\lambda \in P$ and $\gamma$ is an automorphism of $\g$, then $\rho \circ \gamma^{-1}$ is the irreducible representation of $\g$ of highest weight $\gamma_{\mathrm{Out}} \cdot \lambda$, where $\gamma_{\mathrm{Out}}$ is the outer part of the automorphism $\gamma$ (see \cite[VIII, \S7.2, Rem.~1]{Bou75}).  So the group $\Gamma$ acts naturally on each $P^+$ via the quotient $\Aut \g \twoheadrightarrow \Out \g$.

Let $\cE$ denote the set of finitely supported functions $\psi : X_\rat \to P^+$ and let $\cE^\Gamma$ denote the subset of $\cE$ consisting of those functions which are $\Gamma$-equivariant. Here the support of $\psi \in \cE$ is
\[
  \Supp \psi = \{x \in X_\rat\ |\ \psi(x) \ne 0\}.
\]

If $\bx \in X_*$ and $\rho_x,\rho_x'$ are isomorphic representations of $\g$ for each $x \in \bx$, the evaluation representations $\ev^\Gamma_\bx (\rho_x)_{x \in \bx}$ and $\ev^\Gamma_\bx (\rho_x')_{x \in \bx}$ are isomorphic. Therefore, for $\bx \in X_*$ and representations $\rho_x$ of $\g^x$ for $x \in \bx$, we define $\ev^\Gamma_\bx ([\rho_x])_{x \in \bx}$ to be the isomorphism class of $\ev^\Gamma_\bx (\rho_x)_{x \in \bx}$. (Here $[\rho_x]$, $x \in \bx$, denotes the isomorphism class of the representation $\rho_x$.)

For $\psi \in \cE^\Gamma$, we define $\ev^\Gamma_\psi = \ev^\Gamma_\bx (\psi(x))_{x \in \bx}$, where $\bx \in X_*$ contains one element of each $\Gamma$-orbit in $\Supp \psi$.  By \cite[Lem.~4.13]{NSS}, $\ev^\Gamma_\psi$ is independent of the choice of $\bx$.  If $\psi$ is the map that is identically 0 on $X$, we define $\ev^\Gamma_\psi$ to be the isomorphism class of the trivial representation of $(\g \otimes A)^\Gamma$.  We say that an evaluation representation is a \emph{single orbit evaluation representation} if its isomorphism class is $\ev^\Gamma_\psi$ for some $\psi \in \cE^\Gamma$ whose support is contained in a single $\Gamma$-orbit.  For all of the above notation, we drop the superscript $\Gamma$ when $\Gamma = \{1\}$.  For instance, for a finite subset $\bx \subseteq X_\rat$, $\ev_\bx : \g \otimes A \to \g^\bx$ is the corresponding evaluation map.  Similarly, for $\psi \in \cE$, $\ev_\psi$ is the corresponding isomorphism class of representations of $\g \otimes A$.

\begin{proposition}[{\cite[Th.~5.5]{NSS}}] \label{prop:irred-classification}
  The map
  \[
    \cE^\Gamma \to \cS^\Gamma,\quad \psi \mapsto \ev^\Gamma_\psi,\quad \psi \in \cE^\Gamma,
  \]
  is a bijection, where $\cS^\Gamma$ denotes the set of isomorphism classes of irreducible finite-dimensional representations of $(\g \otimes A)^\Gamma$.  In particular, all irreducible finite-dimensional representations of $(\g \otimes A)^\Gamma$ are evaluation
  representations.
\end{proposition}

\begin{remark} \label{rem:irred-classification}
  The classification of irreducible finite-dimensional representations given in \cite{NSS} is much more general than Proposition~\ref{prop:irred-classification}.  In particular, it applies in the case that $\g$ is any finite-dimensional Lie algebra, $\Gamma$ is any finite group (i.e.\ not necessarily abelian), and the action of $\Gamma$ is arbitrary (i.e.\ $\Gamma$ need not act freely on $X$).  In this generality, all irreducible finite-dimensional representations are tensor products of evaluation representations and one-dimensional representations.  However, under the more restrictive assumptions of the current paper, $(\g \otimes A)^\Gamma$ is a perfect Lie algebra (i.e.\ $[(\g \otimes A)^\Gamma,(\g \otimes A)^\Gamma]=(\g \otimes A)^\Gamma$) and so $(\g \otimes A)^\Gamma$ has no nontrivial one-dimensional representations, \cite[Lem.~6.1]{NS11}.
\end{remark}

\begin{definition}[Notation for irreducibles] \label{def:irred-notation}
  For $\psi \in \mathcal{E}^\Gamma$, we let $V_\Gamma(\psi)$ denote the corresponding irreducible representation of $(\g \otimes A)^\Gamma$ (that is, $V_\Gamma(\psi)$ is some irreducible representation in the isomorphism class $\ev^\Gamma_\psi$).  For $\psi \in \mathcal{E}$, we let $V(\psi)$ denote the corresponding irreducible representation of $\g \otimes A$.
\end{definition}

\begin{example}[Untwisted map algebras]
  When the group $\Gamma$ is trivial, $(\g \otimes A)^\Gamma = \g \otimes A$ is called an \emph{untwisted map algebra}, or \emph{generalized current algebra}.  These algebras arise also for a nontrivial group $\Gamma$ acting trivially on $\g$ or on $X$. In the first case we have $(\g \otimes A)^\Gamma \cong \g \otimes A^\Gamma$, and in the second
  $(\g \otimes A)^\Gamma = \g^\Gamma \otimes A$.
\end{example}

\begin{example}[Multiloop algebras] \label{eg:multiloop}
  Fix positive integers $n, m_1, \dots, m_n$.  Let
  \[
    \Gamma = \langle \gamma_1,\dots, \gamma_n : \gamma_i^{m_i}=1,\ \gamma_i \gamma_j = \gamma_j \gamma_i,\ \forall\ 1 \le i,j \le n \rangle \cong \Z/m_1 \Z \times \dots \times \Z/m_n \Z,
  \]
  and suppose that $\Gamma$ acts on $\g$. Note that this is equivalent to specifying commuting automorphisms $\sigma_i$, $i=1,\dots,n$, of $\g$ such that $\sigma_i^{m_i}=\id$. For $i = 1,\dots, n$, let $\xi_i$ be a primitive $m_i$-th root of unity. Let $X=(k^\times)^n$ and define an action of $\Gamma$ on $X$ by
  \[
    \gamma_i \cdot (z_1, \dots, z_n) = (z_1, \dots, z_{i-1}, \xi_i z_i, z_{i+1}, \dots, z_n).
  \]
  Then
  \begin{equation} \label{eq:twisted-multiloop-def}
      M(\g,\sigma_1,\dots,\sigma_n,m_1,\dots,m_n) := (\g \otimes A)^\Gamma
  \end{equation}
  is the \emph{multiloop algebra} of $\g$ relative to $(\sigma_1,
  \dots, \sigma_n)$ and $(m_1, \ldots, m_n)$.
\end{example}

\begin{definition}[$\g$-weights]
  We can identify $\g$ with the subalgebra $\g \otimes k \subseteq \g \otimes A$.  In this way, any $(\g \otimes A$)-module $V$ can be viewed as a $\g$-module.  We will refer to the weights of this $\g$-module as the \emph{$\g$-weights} of $V$ (assuming $V$ has a weight decomposition, e.g.\ $V$ is finite-dimensional).  For a $\g$-weight $\lambda$, we let $V_\lambda$ denote the corresponding weight space of $V$.
\end{definition}

%%%%%%%%%%%%%%%%%%%%%%%%%%%%%%%%%%%%%%%%%%%%%%%%%%%%%%%%%%%%%%%%%%%
%
\section{Twisting and untwisting functors} \label{sec:cat-equivalence}
%
%%%%%%%%%%%%%%%%%%%%%%%%%%%%%%%%%%%%%%%%%%%%%%%%%%%%%%%%%%%%%%%%%%%

In this section, we define isomorphisms between certain categories of modules for (untwisted) map algebras $\g \otimes A$ and their equivariant analogues $(\g \otimes A)^\Gamma$.  This isomorphism will be our key tool in defining local Weyl modules in the equivariant setting and proving their characterization via homological properties.

Recall that for a point $x \in X_\rat$, $\fm_x$ denotes the corresponding maximal ideal of $A$.  For $\eta : X_\rat \to \N = \Z_{\ge 0}$ with finite support, define
\begin{equation} \label{eq:I_eta-def} \ts
  I_\eta = \prod_{x \in \Supp \eta} \fm_x^{\eta(x)}.
\end{equation}
For a finite subset $\bx \subseteq X$, we define $I_\bx = I_\eta$, where $\eta(x) = 1$ for $x \in \bx$ and $\eta(x) = 0$ for $x \not \in \bx$.  It is straightforward to check that $\g \otimes I_\eta$ is an ideal of $\g \otimes A$ and so we have a \emph{generalized evaluation map}
\begin{gather*}
  \ts \ev_\eta : \g \otimes A \twoheadrightarrow (\g \otimes A)/(\g \otimes I_\eta) \cong \bigoplus_{x \in \Supp \eta} \g \otimes (A/\fm_x^{\eta(x)}) \cong \bigoplus_{x \in \Supp \eta} (\g \otimes A)/(\g \otimes \fm_x^{\eta(x)}), \\
  \ts \ev_\eta(\alpha) = \bigoplus_{x \in \Supp \eta} (\alpha + (\g \otimes \fm_x^{\eta(x)})).
\end{gather*}
Let
\[ \ts
  \ev_\eta^\Gamma : (\g \otimes A)^\Gamma \to \bigoplus_{x \in \Supp \eta} (\g \otimes A)/(\g \otimes \fm_x^{\eta(x)})
\]
denote the restriction of $\ev_\eta$ to $(\g \otimes A)^\Gamma$.  Clearly
\[
  \ker \ev_\eta^\Gamma = (\ker \ev_\eta) \cap (\g \otimes A)^\Gamma = (\g \otimes I_\eta) \cap (\g \otimes A)^\Gamma = (\g \otimes I_\eta)^\Gamma.
\]

Recall that $X_*$ is the set of finite subsets of $X_\rat$ that do not contain two points in the same $\Gamma$-orbit.

\begin{lemma} \label{lem:ideal-equality}
  If $\eta : X_\rat \to \N$ satisfies $\Supp \eta \in X_*$, then
  \[ \ts
    (\g \otimes I_\eta)^\Gamma = (\g \otimes \tilde I_\eta)^\Gamma, \quad \text{where} \quad
    \tilde I_\eta = \prod_{x \in \Supp \eta} \prod_{\gamma \in \Gamma} \fm_{\gamma \cdot x}^{\eta(x)}.
  \]
\end{lemma}

\begin{proof}
  Since $\tilde I_\eta \subseteq I_\eta$, we have $(\g \otimes \tilde I_\eta)^\Gamma \subseteq (\g \otimes I_\eta)^\Gamma$.  Suppose $\alpha \in (\g \otimes I_\eta)^\Gamma$.  Then for each $x \in \Supp \eta$ and $\gamma \in \Gamma$, we have
  \[
    \alpha \in (\g \otimes I_\eta)^\Gamma \subseteq \g \otimes I_\eta \subseteq \g \otimes \fm_x^{\eta(x)} \implies \alpha = \gamma \cdot \alpha \in \gamma (\g \otimes \fm_x^{\eta(x)}) = \g \otimes \fm_{\gamma \cdot x}^{\eta(x)}.
  \]
  Thus
  \[ \ts
    (\g \otimes I_\eta)^\Gamma \subseteq \g \otimes \bigcap_{x \in \Supp \eta} \bigcap_{\gamma \in \Gamma} \fm_{\gamma \cdot x}^{\eta(x)} = \g \otimes \tilde I_\eta,
  \]
  since the ideals $\fm_{\gamma \cdot x}$ are relatively prime.  Thus $(\g \otimes I_\eta)^\Gamma \subseteq (\g \otimes \tilde I_\eta)^\Gamma$.
\end{proof}

\begin{proposition} \label{prop:fundamental-quotient-isom}
  If $\eta : X_\rat \to \N$ satisfies $\Supp \eta \subseteq X_*$, then the map $\ev_\eta^\Gamma$ is surjective and hence induces an isomorphism
  \[
    (\g \otimes A)^\Gamma/(\g \otimes I_\eta)^\Gamma \xrightarrow{\cong} (\g \otimes A)/(\g \otimes I_\eta).
  \]
\end{proposition}

\begin{proof}
  It suffices to show that for arbitrary $a \in \g$, $f \in A$, $x \in \Supp \eta$, there exists $\alpha \in (\g \otimes A)^\Gamma$ such that
  \[
    \alpha - (a \otimes f) \in \g \otimes \fm_x^{\eta(x)},\quad \alpha \in \g \otimes \fm_y^{\eta(y)} \ \forall\ y \in \Supp \eta \setminus \{x\}.
  \]
  Let $n = \max_{y \in \Supp \eta} \eta(y)$ and let $\xi$ be an $n$-th root of $-1$.  Since the action of $\Gamma$ on $X$ is free, we can choose $f_1 \in A$ such that
  \[
    f_1(x) = 0, \quad f_1(\gamma \cdot x)=\xi \ \forall\ \gamma \in \Gamma, \gamma \ne 1,\quad f_1(\gamma \cdot y) = \xi \ \forall\ \gamma \in \Gamma, y \in \Supp \eta \setminus \{x\}.
  \]
  Then $f_1 \in \fm_x$.  So
  \[
    f_1^n \in \fm_x^n,\quad f_1^n(\gamma \cdot x)=-1 \ \forall\ \gamma \in \Gamma, \gamma \ne 1,\quad f_1^n(\gamma \cdot y) = -1 \ \forall\ \gamma \in \Gamma, y \in \Supp \eta \setminus \{x\}.
  \]
  Hence
  \[ \ts
    1 + f_1^n \in 1 + \fm_x^n,\quad 1+f_1^n \in \prod_{\gamma \in \Gamma,\, \gamma \ne 1} \fm_{\gamma \cdot x} \prod_{\substack{\gamma \in \Gamma \\ y \in \Supp \eta \setminus \{x\}}} \fm_{\gamma \cdot y}.
  \]
  Recall that for any ideal $I$ of $A$, the set $1 + I$ is closed under multiplication.  Thus
  \[ \ts
    (1+f_1^n)^n \in 1 + \fm_x^n,\quad (1+f_1^n)^n \in \prod_{\gamma \in \Gamma,\, \gamma \ne 1} \fm_{\gamma \cdot x}^n \prod_{\substack{\gamma \in \Gamma \\ y \in \Supp \eta \setminus \{x\}}} \fm_{\gamma \cdot y}^n,
  \]
  and so, setting $f_2 = f(1+f_1^n)^n$, we have
  \[ \ts
    f_2 \in f + \fm_x^n,\quad f_2 \in \prod_{\gamma \in \Gamma,\, \gamma \ne 1} \fm_{\gamma \cdot x}^n \prod_{\substack{\gamma \in \Gamma \\ y \in \Supp \eta \setminus \{x\}}} \fm_{\gamma \cdot y}^n.
  \]
  Define
  \[ \ts
    \alpha = \sum_{\gamma \in \Gamma} \gamma \cdot (a \otimes f_2) = \sum_{\gamma \in \Gamma} (\gamma \cdot a) \otimes (\gamma \cdot f_2) \in (\g \otimes A)^\Gamma.
  \]
  Since $\gamma \cdot \fm_y = \fm_{\gamma \cdot y}$ and $\Gamma$ acts freely on $X$, we have
  \[
    \gamma \cdot f_2 \in \fm_x^n \subseteq \fm_x^{\eta(x)} \ \forall\ \gamma \in \Gamma, \gamma \ne 1.
  \]
  Thus
  \[
    \alpha + \g \otimes \fm_x^{\eta(x)} = (a \otimes f_2) + \g \otimes \fm_x^{\eta(x)} = a \otimes f + \g \otimes \fm_x^{\eta(x)}.
  \]
  We also have
  \[
    \gamma \cdot f_2 \in \fm_y^n \subseteq \fm_y^{\eta(y)} \ \forall\ \gamma \in \Gamma,\ y \in \Supp \eta \setminus \{x\},
  \]
  and so
  \[
    \alpha \in \g \otimes \fm_y^{\eta(y)} \ \forall\ y \in \Supp \eta \setminus \{x\}.
  \]
\end{proof}

Let $\Xi$ be the character group of $\Gamma$. This is an abelian group, whose group operation we will write additively. Hence, $0$ is the character of the trivial one-dimensional representation, and if an irreducible representation affords the character $\xi$, then $-\xi$ is the character of the dual representation.

If $\Gamma$ acts on an algebra $B$ by automorphisms, it is well known that $B=\bigoplus_{\xi \in \Xi} B_\xi$ is a $\Xi$-grading, where $B_\xi$ is the isotypic component of type $\xi$. It follows that $(\g \otimes A)^\Gamma$ can be written as
\begin{equation} \label{eq:M-grading}
  (\g \otimes A)^\Gamma = \ts \bigoplus_{\xi \in \Xi} \, \g_\xi \otimes A_{-\xi},
\end{equation}
since $\g = \bigoplus_\xi \g_\xi$ and $A=\bigoplus_\xi A_\xi$ are $\Xi$-graded and $(\g_\xi \otimes A_{\xi'})^\Gamma = 0$ if $\xi'\ne -\xi$. The decomposition \eqref{eq:M-grading} is an algebra $\Xi$-grading.

\begin{lemma}[{\cite[Lem.~4.4]{NS11}}] \label{lem:free-actions}
  Suppose a finite abelian group $\Gamma$ acts on a unital associative commutative $k$-algebra $A$ (and hence on $X = \Spec A$) by automorphisms.  Let $A = \bigoplus_{\xi \in \Xi} A_\xi$ be the associated grading on $A$, where $\Xi$ is the character group of $\Gamma$.  Then the following conditions are equivalent:
  \begin{enumerate}
    \item \label{lem-item:free-actions:Gamma-free} $\Gamma$ acts freely on $X$, and
    \item \label{lem-item:free-actions:ideal-strong-grading} $\prod_{i=1}^n I_{\xi_i} = (I^n)_{\sum_{i=1}^n \xi_i}$ for all $\xi_1,\dots,\xi_n \in \Xi$ and any $\Gamma$-invariant ideal $I$ of $A$.  Here $I_\xi = I \cap A_\xi$ for $\xi \in \Xi$.
  \end{enumerate}
\end{lemma}

For a Lie algebra $L$, define $L^n$, $n \ge 1$, by
\[
  L^1 = L,\quad L^n = [L,L^{n-1}],\quad n > 1.
\]
The following proposition, combined with Proposition~\ref{prop:fundamental-quotient-isom}, will allow us to define and deduce properties of finite-dimensional modules for equivariant map algebras from the corresponding notions for untwisted map algebras.

\begin{proposition} \label{prop:fd-annihilator}
  Every finite-dimensional $(\g \otimes A)^\Gamma$-module is annihilated by $(\g \otimes I_\eta)^\Gamma$ for some $\eta : X_\rat \to \N$ with $\Supp \eta \subseteq X_*$.
\end{proposition}

\begin{proof}
  Suppose $V$ is a finite-dimensional $(\g \otimes A)^\Gamma$-module annihilated by $(\g \otimes I_\eta)^\Gamma$ for some finitely supported $\eta : X_\rat \to \N$.  By Lemma~\ref{lem:ideal-equality}, we can find $\eta' : X_\rat \to \N$ with $\Supp \eta' \subseteq X_*$ and $(\g \otimes I_{\eta'})^\Gamma \subseteq (\g \otimes I_\eta)^\Gamma$.  Thus it suffices to prove that every finite-dimensional $(\g \otimes A)^\Gamma$-module is annihilated by some $(\g \otimes I_\eta)^\Gamma$.

  We first prove by induction that for any $\Gamma$-invariant ideal $I$ of $A$,
  \begin{equation} \label{eq:power-ideal}
    \left( (\g \otimes I)^\Gamma \right)^m = (\g \otimes I^m)^\Gamma \quad \forall\ m \ge 1.
  \end{equation}
  The result is trivial for $m=1$.  Assume it is true for some $m \ge 1$.  Then
  \begin{flalign*}
    && \left( (\g \otimes I)^\Gamma \right)^{m+1} &= \left[ (\g \otimes I)^\Gamma, \left( (\g \otimes I)^\Gamma \right)^m \right] \\
    && &= \left[ (\g \otimes I)^\Gamma, (\g \otimes I^m)^\Gamma \right] && \text{(by the induction hypothesis)} \\
    && &= \ts \left[ \bigoplus_{\xi \in \Xi} \g_\xi \otimes I_{-\xi}, \bigoplus_{\tau \in \Xi} \g_\tau \otimes (I^m)_{-\tau} \right] \\
    && &= \ts \sum_{\xi,\tau \in \Xi} [\g_\xi,\g_\tau] \otimes I_{-\xi} (I^m)_{-\tau} \\
    && &= \ts \sum_{\xi,\tau \in \Xi} [\g_\xi,\g_\tau] \otimes (I^{m+1})_{-\xi-\tau} && \text{(by Lemma~\ref{lem:free-actions})} \\
    && &= \ts \bigoplus_{\xi \in \Xi} \g_\xi \otimes (I^{m+1})_{-\xi} && \text{(since $\g$ is semisimple)} \\
    && &= (\g \otimes I^{m+1})^\Gamma.
  \end{flalign*}
  Thus \eqref{eq:power-ideal} holds.

  Now let $V$ be a finite-dimensional $(\g \otimes A)^\Gamma$-module.  Then there exists a filtration
  \[
    0 = V_0 \subseteq V_1 \subseteq \dots \subseteq V_n = V,
  \]
  such that $V_i/V_{i-1}$ is an irreducible finite-dimensional $(\g \otimes A)^\Gamma$-module for $1 \le i \le n$.  By Proposition~\ref{prop:irred-classification}, each $V_i/V_{i-1}$ is an evaluation module.  Let  $\eta_i : X_\rat \to \N$ be the characteristic function of the support of $V_i/V_{i-1}$.  Then $(\g \otimes I_{\eta_i})^\Gamma \cdot (V_i/V_{i-1}) = 0$.  In other words, $(\g \otimes I_{\eta_i})^\Gamma \cdot V_i \subseteq V_{i-1}$.

  Let $\nu = \sum_{i=1}^n \eta_i$ and $\eta = n \nu$.  We claim that $(\g \otimes I_\eta)^\Gamma \cdot V = 0$.  Since $I_\eta = I_\nu^n$, it follows from \eqref{eq:power-ideal} that $\left( (\g \otimes I_\nu)^\Gamma \right)^n = (\g \otimes I_\eta)^\Gamma$.  Because $I_\nu \subseteq I_{\eta_i}$, we have $(\g \otimes I_\nu)^\Gamma \cdot V_i \subseteq V_{i-1}$ for all $1 \le i \le n$.  Therefore
  \[
    (\g \otimes I_\eta)^\Gamma \cdot V = \left( (\g \otimes I_\nu)^\Gamma \right)^n \cdot V = 0.\qedhere
  \]
\end{proof}

For functions $\eta, \eta' : X_\rat \to \N$ with finite support, we write $\eta \le \eta'$ if $\eta(x) \le \eta'(x)$ for all $x \in X_\rat$.  Clearly
\[
  \eta \le \eta' \implies I_\eta \supseteq I_{\eta'} \implies \g \otimes I_\eta \supseteq \g \otimes I_{\eta'}.
\]
Thus, for $\eta \le \eta'$, we have natural projections
\[
  (\g \otimes A)/(\g \otimes I_{\eta'}) \twoheadrightarrow (\g \otimes A)/(\g \otimes I_{\eta}),
  \quad (\g \otimes A)^\Gamma/(\g \otimes I_{\eta'})^\Gamma \twoheadrightarrow (\g \otimes A)^\Gamma/(\g \otimes I_{\eta})^\Gamma.
\]

\begin{lemma} \label{lem:ideal-isom-comm-diagram}
  If $\eta,\eta' : X_\rat \to \N$ are such that $\eta \le \eta'$ and $\Supp \eta' \subseteq X_*$, then the diagram
  \begin{equation} \label{eq:ideal-isom-comm-diagram}
    \xymatrix{
      (\g \otimes A)^\Gamma/(\g \otimes I_{\eta'})^\Gamma \ar[r]^{\cong} \ar@{->>}[d] & (\g \otimes A)/(\g \otimes I_{\eta'}) \ar@{->>}[d] \\
      (\g \otimes A)^\Gamma/(\g \otimes I_{\eta})^\Gamma \ar[r]^{\cong} & (\g \otimes A)/(\g \otimes I_{\eta})
    }
  \end{equation}
  is commutative, where the horizontal maps are the isomorphisms induced by evaluation as in Proposition~\ref{prop:fundamental-quotient-isom}.
\end{lemma}

\begin{proof}
  This is clear from the fact that both compositions in the diagram are induced from the composition
  \[
    (\g \otimes A)^\Gamma \hookrightarrow \g \otimes A \twoheadrightarrow (\g \otimes A)/(\g \otimes I_\eta).
  \]
\end{proof}

Suppose $V$ is a finite-dimensional $(\g \otimes A)^\Gamma$-module.  By Proposition~\ref{prop:fd-annihilator}, there exists a function $\eta : X_\rat \to \N$, $\Supp \eta \subseteq X_*$, such that $(\g \otimes I_\eta)^\Gamma$ annihilates $V$.  Therefore the action of $(\g \otimes A)^\Gamma$ on $V$ factors through $(\g \otimes A)^\Gamma/(\g \otimes I_\eta)^\Gamma$ and the composition
\[
  \g \otimes A \twoheadrightarrow (\g \otimes A)/(\g \otimes I_\eta) \cong (\g \otimes A)^\Gamma/(\g \otimes I_\eta)^\Gamma \to \End V
\]
defines an action of $(\g \otimes A)$ on $V$.  We denote the resulting $(\g \otimes A)$-module by $V^\eta$.

\begin{lemma} \label{lem:eta-independence}
  Suppose $V$ is a finite-dimensional $(\g \otimes A)^\Gamma$-module that is annihilated by $(\g \otimes I_\eta)^\Gamma$ and $(\g \otimes I_{\eta'})^\Gamma$ for functions $\eta,\eta' : X_\rat \to \N$ such that $\Supp \eta \cup \Supp \eta' \subseteq X_*$.  Then $V^\eta = V^{\eta'}$ as $(\g \otimes A)$-modules.
\end{lemma}

\begin{proof}
  Let $\tau = \eta + \eta'$.  It is clear that $(\g \otimes I_\tau)^\Gamma$ annihilates $V$.  Since $\Supp \tau = \Supp \eta \cup \Supp \eta'$, it follows from Lemma~\ref{lem:ideal-isom-comm-diagram} that the diagram
  \[
    \xymatrix{
      & (\g \otimes A)/(\g \otimes I_\eta) \ar[r]^(0.46){\cong} & (\g \otimes A)^\Gamma/(\g \otimes I_{\eta})^\Gamma \ar[dr] & \\
      \g \otimes A \ar@{->>}[ur] \ar@{->>}[r] \ar@{->>}[dr] & (\g \otimes A)/(\g \otimes I_\tau) \ar[r]^(0.46){\cong} \ar@{->>}[u] \ar@{->>}[d] & (\g \otimes A)^\Gamma/(\g \otimes I_\tau)^\Gamma \ar[r] \ar@{->>}[u] \ar@{->>}[d] & \End V \\
      & (\g \otimes A)/(\g \otimes I_{\eta'}) \ar[r]^(0.46){\cong} & (\g \otimes A)^\Gamma/(\g \otimes I_{\eta'})^\Gamma \ar[ur] &
    }
  \]
  commutes, where the three isomorphisms in the middle are the inverses of the isomorphisms of Proposition~\ref{prop:fundamental-quotient-isom} induced by evaluation.  It follows that $V^\eta = V^\tau = V^{\eta'}$ as $(\g \otimes A)$-modules.
\end{proof}

\begin{definition}[Categories $\cF$, $\cF^\Gamma$, $\cF_\bx$, and $\cF_\bx^\Gamma$]
  Let $\cF$ and $\cF^\Gamma$ be the categories of finite-dimensional representations of $\g \otimes A$ and $(\g \otimes A)^\Gamma$ respectively.  For $\bx \in X_*$, define $\cF_\bx$ (resp.\ $\cF_\bx^\Gamma$) to be the full subcategory of $\cF$ (resp.\ $\cF^\Gamma$) consisting of those representations whose irreducible constituents all have support contained in $\bx$ (resp.\ $\Gamma \cdot \bx$).
\end{definition}

\begin{definition}[Twisting functor]
  We have a natural \emph{twisting functor} $\bT : \cF \to \cF^\Gamma$ defined by restricting from $\g \otimes A$ to $(\g \otimes A)^\Gamma$.  For any $\bx \in X_*$, we have the induced functor $\bT_\bx : \cF_\bx \to \cF_\bx^\Gamma$.
\end{definition}

\begin{definition}[Untwisting functor]
  Fix $\bx \in X_*$.  By Proposition~\ref{prop:fd-annihilator}, every module $V \in \cF_\bx^\Gamma$ is annihilated by some $(\g \otimes I_\eta)^\Gamma$ with $\Supp \eta \subseteq \bx$.  By Lemma~\ref{lem:eta-independence}, the modules $V^\eta$ are independent of the choice of $\eta$.  The \emph{untwisting functor} $\bU_\bx : \cF_\bx^\Gamma \to \cF_\bx$ is defined to be the functor that, on objects, maps $V$ to $V^\eta$.  Now suppose $V, W \in \cF_\bx^\Gamma$ and $\beta : V \to W$ is a morphism in $\cF_\bx^\Gamma$.  Since $\cF_\bx^\Gamma$ is a full subcategory of $\cF^\Gamma$, $\beta : V \to W$ is a morphism in $\cF^\Gamma$, which means that it is a homomorphism of $(\g \otimes A)^\Gamma$-modules.  Choose $\eta : X_\rat \to \N$ with support contained in $\bx$ such that $(\g \otimes I_\eta)^\Gamma$ annihilates both $V$ and $W$.  Then the action of $(\g \otimes A)^\Gamma$ on $V$ and $W$ factors through $(\g \otimes A)^\Gamma/(\g \otimes I_\eta)^\Gamma$.  By definition, it follows that $\beta$ is also a homomorphism of $(\g \otimes A)$-modules from $V^\eta$ to $W^\eta$.  We define $\bU_\bx(\beta)$ to be this homomorphism.  One easily sees that $\bU_\bx$ respects composition of morphisms and hence is a well-defined functor.
\end{definition}

For a $\Gamma$-invariant subset $Y$ of $X_\rat$, let $Y_\Gamma$ denote the set of subsets of $Y$ containing exactly one point from each $\Gamma$-orbit in $Y$.  For $\psi \in \cE^\Gamma$ and $\bx \in (\Supp \psi)_\Gamma$, define
\[
  \psi_\bx : X_\rat \to P^+,\quad \psi_\bx(x) =
  \begin{cases}
    \psi(x) & \text{if } x \in \bx, \\
    0 & \text{if } x \not \in \bx.
  \end{cases}
\]

\begin{theorem} \label{theo:category-equivalence}
  For $\bx \in X_*$, the twisting and untwisting functors have the following properties.
  \begin{enumerate}
    \item \label{theo-item:restrction-functor} The twisting functor $\bT$ maps the isomorphism class $\ev_\psi$ for $\psi \in \cE$, $\Supp \psi \subseteq X_*$, to the isomorphism class $\ev^\Gamma_{\psi^\Gamma}$ for $\psi^\Gamma \in \cE^\Gamma$, where
        \[ \ts
          \psi^\Gamma(x) = \sum_{\gamma \in \Gamma} \gamma \cdot \psi(\gamma^{-1} \cdot x), \quad x \in X_\rat.
        \]

    \item \label{theo-item:untwisting-psi} The untwisting functor $\bU_\bx$ maps the isomorphism class $\ev^\Gamma_\psi$, $\psi \in \cE^\Gamma$, to the isomorphism class $\ev_{\psi_\bx}$.

    \item \label{theo-item:category-equivalence} The functors $\bT_\bx$ and $\bU_\bx$ are mutually inverse isomorphisms of categories.
  \end{enumerate}
\end{theorem}

\begin{proof}
  Part~\eqref{theo-item:restrction-functor} follows immediately from the definition of the evaluation representations given in Section~\ref{sec:ema-background}.

  Now suppose $\bx \in X_*$ and $V \in \cF^\Gamma_\bx$ is irreducible and corresponds to $\psi \in \cE^\Gamma$.  Let $\rho = \left( \bigotimes_{x \in \bx} \rho_x \right) \circ \ev^\Gamma_\bx$ be the corresponding representation.  Then $\rho$ factors through $(\g \otimes A)^\Gamma / (\g \otimes I_\bx)^\Gamma$ and so $\bU_\bx(V)$ is the $(\g \otimes A)$-module given by the composition
  \[ \textstyle
    \g \otimes A \twoheadrightarrow (\g \otimes A)/(\g \otimes I_\bx) \xrightarrow{\cong} (\g \otimes A)^\Gamma/(\g \otimes I_\bx)^\Gamma \cong \g^\bx \xrightarrow{\bigotimes_{x \in \bx} \rho_x} \End V.
  \]
  Since this is precisely the evaluation representation $\left( \bigotimes_{x \in \bx} \rho_x \right) \circ \ev_\bx$ of $\g \otimes A$, which is in the isomorphism class $\ev_{\psi_\bx}$, part~\eqref{theo-item:untwisting-psi} follows.

  Suppose $V \in \cF_\bx$.  Then $V$ is annihilated by some $\g \otimes I_\eta$ and the action of $\g \otimes A$ on $\bU_\bx \bT_\bx (V)$ is given by
  \[
    \g \otimes A \twoheadrightarrow (\g \otimes A)/(\g \otimes I_\eta) \xrightarrow{\cong} (\g \otimes A)^\Gamma/(\g \otimes I_\eta)^\Gamma \xrightarrow{\cong} (\g \otimes A)/(\g \otimes I_\eta) \to \End V,
  \]
  where the two isomorphisms are mutually inverse.  Thus $\bU_\bx \bT_\bx (V) = V$.  One easily verifies that $\bU_\bx \bT_\bx$ is also the identity on morphisms and is therefore the identity functor on $\cF_\bx$.  Similarly, $\bT_\bx \bU_\bx$ is the identity functor on $\cF^\Gamma_\bx$.  This proves part~\eqref{theo-item:category-equivalence}.
\end{proof}

\begin{remark}
  Theorem~\ref{theo:category-equivalence} allows one to translate any reasonable question in the representation theory of finite-dimensional modules for equivariant maps algebras, where $\Gamma$ is abelian and acts freely on $X$, to a corresponding question for untwisted map algebras (generalized current algebras).  For instance, it can be used to reduce the computation of extensions between irreducible finite-dimensional $(\g \otimes A)^\Gamma$-modules to the case of extensions of $(\g \otimes A)$-modules, which were considered in \cite{Kod10}.  In this way, one can give an alternate proof of \cite[Prop.~6.3]{NS11}.
\end{remark}

%%%%%%%%%%%%%%%%%%%%%%%%%%%%%%%%%%%%%%%%%%%%%%%%%%%%%%%%%%%%%%%%%%%
%
\section{Local Weyl modules} \label{sec:local-weyl-modules}
%
%%%%%%%%%%%%%%%%%%%%%%%%%%%%%%%%%%%%%%%%%%%%%%%%%%%%%%%%%%%%%%%%%%%

In this section, we define the local Weyl modules for equivariant map algebras.  We begin by reviewing the local Weyl modules for untwisted map algebras.

Fix a triangular decomposition $\g=\mathfrak{n}^-\oplus \mathfrak h\oplus \mathfrak{n}^+$.  Then we have a triangular decomposition of the untwisted map algebra
\[
  \g \otimes A = (\mathfrak{n}^- \otimes A) \oplus (\mathfrak{h} \otimes A) \oplus (\mathfrak{n}^+ \otimes A).
\]
Let $\{e_i,h_i,f_i\}_{i \in I}$ denote a set of Chevalley generators of $\g$ compatible with its triangular decomposition.  In particular, the $f_i$ generate $\mathfrak{n}^-$.

\begin{definition}[Untwisted local Weyl module] \label{def:untwisted-local-Weyl-module}
  Given $\psi\in \cE$, the \emph{(untwisted) local Weyl module} $W(\psi)$ is the $(\g \otimes A)$-module generated by a nonzero vector $w_\psi$ satisfying the relations
  \begin{gather}
    (\mathfrak{n}^+ \otimes A) \cdot w_\psi =0, \\
    \ts (f_i \otimes 1)^{\lambda(h_i)+1} \cdot w_\psi=0, \quad i\in I, \quad \text{where } \lambda = \wt \psi := \sum_{x\in \bx} \psi(x),\\
    \ts \alpha \cdot w_\psi = \left( \sum_{x\in \Supp \psi} \psi(x)(\alpha(x)) \right) w_\psi, \quad \alpha \in \mathfrak{h} \otimes A.
  \end{gather}
\end{definition}

\begin{proposition} \label{prop:Weyl-untwisted-quotient}
  \begin{enumerate}
    \item \cite[Th.~2]{CFK} \label{prop-item:Weyl-fd} For every $\psi \in \cE$, $W(\psi)$ is a finite-dimensional $(\g \otimes A)$-module.

    \item \cite[Prop.~5]{CFK} \label{prop-item:Weyl-untwisted:quotient} Let $V$ be any finite-dimensional $(\g \otimes A)$-module generated by a nonzero element $v \in V$ such that
        \[
          (\mathfrak{n}^+ \otimes A) \cdot v=0 \quad \text{and} \quad (\mathfrak{h} \otimes A) \cdot v = k v.
        \]
        Then there exists $\psi\in \cE$ such that the assignment $w_\psi \mapsto v$ extends to a surjective homomorphism $W(\psi ) \twoheadrightarrow V$ of $(\g \otimes A)$-modules.
  \end{enumerate}
\end{proposition}

For a subset $Y \subseteq X_\rat$, let
\[
  I_Y = \{f \in A \ |\ f(x) = 0 \ \forall\ x \in Y\}.
\]
For $\psi \in \cE$, we define $I_\psi = I_{\Supp \psi}$.  Note then that $I_\psi = I_\eta$ as in~\eqref{eq:I_eta-def} for
\[
  \eta : X_\rat \to \N,\quad \eta(x) =
  \begin{cases}
    1 & \text{if } x \in \Supp(\psi), \\
    0 & \text{if } x \not \in \Supp(\psi).
  \end{cases}
\]

\begin{proposition} \label{prop:Weyl-untwisted-ann-tensor}
  \begin{enumerate}
    \item \label{prop-item:untwisted-Weyl-annihilator} \cite[Prop.~9]{CFK} If $\psi\in \mathcal{E}$ with $\wt\psi=\lambda\in P^+$, then
      \[
        (\g \otimes I_\psi^N) \cdot W(\psi) =0 \quad \forall\ N \geq \lambda(h_\theta),
      \]
      where $\theta$ is the highest root for $\g$ and $h_\theta$ is the corresponding coroot.

    \item \label{prop-item:Weyl-untwisted:tensor} \cite[Th.~3]{CFK} If $\psi, \psi' \in \mathcal{E}$ such that $\Supp \psi \cap \Supp \psi' = \varnothing$, then
      \[
        W(\psi+\psi') \cong W(\psi) \otimes W(\psi')
      \]
      as $(\g \otimes A)$-modules.

    \item \cite[Lem.~6]{CFK} For $\psi \in \mathcal{E}$, $V(\psi)$ is the unique irreducible quotient of $W(\psi)$ (see Definition~\ref{def:irred-notation}).
  \end{enumerate}
\end{proposition}

\begin{remark}
  In the case that $A$ is the coordinate algebra of an affine algebraic variety, Proposition~\ref{prop:Weyl-untwisted-quotient} and parts~\eqref{prop-item:untwisted-Weyl-annihilator} and~\eqref{prop-item:Weyl-untwisted:tensor} of Proposition~\ref{prop:Weyl-untwisted-ann-tensor} are proven in \cite{FL04} (Theorems 1, 2, and 5, and Proposition~7).
\end{remark}

We now turn our attention to the equivariant map algebras.  For a $(\g \otimes A)$-module $U$, let $\rho_U : \g \otimes A \to \End_k U$ be the corresponding representation.

\begin{lemma} \label{lem:Weyl-module-automs}
  Suppose $\psi \in \cE^\Gamma$ and $\bx \in (\Supp \psi)_\Gamma$.  Then, for $\gamma \in \Gamma$,
  \[
    \rho_{W(\psi_\bx)} \circ \gamma^{-1} \cong \rho_{W(\psi_{\gamma \cdot \bx})},
  \]
  where $\gamma \cdot \bx = \{ \gamma \cdot x\ |\ x \in \bx\}$.
\end{lemma}

\begin{proof}
  Let $W(\psi_\bx)^\gamma$ be the $(\g \otimes A)$-module corresponding to the representation $\rho_{W(\psi_\bx)} \circ \gamma^{-1}$.  Recall that we identify $\g$ with the subalgebra $\g \otimes k$ of $\g \otimes A$.  Thus, via restriction, we can view $W(\psi_\bx)$ and $W(\psi_\bx)^\gamma$ as $\g$-modules.  Recall that $W(\psi_\bx)$ is a finite-dimensional $\g$-module with $\wt W(\psi_\bx) \subseteq \lambda - Q_+$, where $\lambda = \sum_{x \in \bx} \psi(x)$.  It follows that $W(\psi_\bx)^\gamma$ is a finite-dimensional $\g$-module with $\wt W(\psi_\bx)^\gamma \subseteq \gamma \cdot \lambda - Q_+$.  Furthermore, the $\gamma \cdot \lambda$ weight space of $W(\psi_\bx)^\gamma$ is one-dimensional since the $\lambda$ weight space of $W(\psi_\bx)$ is one-dimensional.

  We also know that $W(\psi_\bx)$ has unique irreducible quotient $V(\psi_\bx)$.  By the definition of $\cE^\Gamma$, we have that $\rho_{V(\psi_{\bx})} \cdot \gamma^{-1} \cong \rho_{V({\psi_{\gamma \cdot \bx}})}$.  Thus $W(\psi_\bx)^\gamma$ has unique irreducible quotient $V(\psi_{\gamma \cdot \bx})$.  Let $v \in W(\psi_\bx)^\gamma$ be a nonzero vector of weight $\gamma \cdot \lambda$ and let $U$ be the smallest $(\g \otimes A)$-submodule of $W(\psi_\bx)^\gamma$ containing $v$.  If $U \ne W(\psi_\bx)^\gamma$, then $U$ is contained in the unique maximal submodule of $W(\psi_\bx)^\gamma$.  But this contradicts the fact that the unique irreducible quotient of $W(\psi_\bx)^\gamma$ has a nonzero $\gamma \cdot \lambda$ weight space.  Therefore $U=W(\psi_\bx)^\gamma$ and so $v$ is a cyclic vector.  It then follows from Proposition~\ref{prop:Weyl-untwisted-quotient}\eqref{prop-item:Weyl-untwisted:quotient} that $W(\psi_\bx)^\gamma$ is isomorphic to a quotient of $W(\psi_{\gamma \cdot \bx})$.  By symmetry, $W(\psi_{\gamma \cdot \bx})$ is also isomorphic to a quotient of $W(\psi_\bx)^\gamma$.  Since these modules are finite-dimensional, we conclude that $W(\psi_\bx)^\gamma \cong W(\psi_{\gamma \cdot x})$.
\end{proof}

\begin{proposition} \label{prop:Weyl-restriction-independence}
  Suppose $\psi \in \cE^\Gamma$ and $\bx, \bx' \in (\Supp \psi)_\Gamma$.  Then the restriction to $(\g \otimes A)^\Gamma$-modules of the Weyl modules $W(\psi_\bx)$ and $W(\psi_{\bx'})$ for $\g \otimes A$ are isomorphic (as $(\g \otimes A)^\Gamma$-modules).
\end{proposition}

\begin{proof}
  We first prove the result in the case that the support of $\psi$ consists of a single $\Gamma$-orbit.  Suppose $x,x' \in \Supp \psi$.  Then there exist a unique $\gamma \in \Gamma$ such that $x' = \gamma \cdot x$.  By Lemma~\ref{lem:Weyl-module-automs}, we have
  \[
    \rho_{W(\psi_x)} \circ \gamma^{-1} \cong \rho_{W(\psi_{x'})}.
  \]
  Since the restriction of the automorphism $\gamma^{-1}$ to $(\g \otimes A)^\Gamma$ is trivial, it follows immediately that the restrictions of $\rho_{W(\psi_x)}$ and $\rho_{W(\psi_{x'})}$ to $(\g \otimes A)^\Gamma$ are isomorphic.  The general result where the support of $\psi$ is a union of $\Gamma$-orbits now follows from Proposition~\ref{prop:Weyl-untwisted-ann-tensor}\eqref{prop-item:Weyl-untwisted:tensor}.
\end{proof}

\begin{definition}[Twisted local Weyl module]
  For $\psi \in \cE^\Gamma$, we define $W_\Gamma(\psi)$ to be the restriction to $(\g \otimes A)^\Gamma$-modules of the Weyl module $W(\psi_\bx)$ for $\g \otimes A$, for some choice of $\bx \in (\Supp \psi)_\Gamma$.  In other words, $W_\Gamma(\psi) := \bT (W(\psi_\bx))$.  By Proposition~\ref{prop:Weyl-restriction-independence}, $W_\Gamma(\psi)$ is independent of the choice of $\bx$ (up to isomorphism).  We call $W_\Gamma(\psi)$ the \emph{(twisted) local Weyl module} of $(\g \otimes A)^\Gamma$ associated to $\psi$.
\end{definition}

\begin{lemma} \label{lem:Weyl-simple-correspondences}
  For $\psi \in \mathcal{E}^\Gamma$ and $\bx \in (\Supp \psi)_\Gamma$, we have $\bU_\bx(W_\Gamma(\psi)) = W(\psi_\bx)$ and $\bU_\bx(V_\Gamma(\psi)) = V(\psi_\bx)$.
\end{lemma}

\begin{proof}
  By definition, $W_\Gamma(\psi) = \bT_\bx(W(\psi_\bx))$.  Thus, by Theorem~\ref{theo:category-equivalence}, we have
  \[
    \bU_\bx(W_\Gamma(\psi)) = \bU_\bx \bT_\bx (W(\psi_\bx)) = W(\psi_\bx).
  \]
  The proof of the second statement is analogous (see the proof of Theorem~\ref{theo:category-equivalence}\eqref{theo-item:untwisting-psi}).
\end{proof}

\begin{proposition}[Tensor product property]
  If $\psi,\psi' \in \mathcal{E}^\Gamma$ have disjoint support, then $W_\Gamma(\psi + \psi') \cong W_\Gamma(\psi) \otimes W_\Gamma(\psi')$.
\end{proposition}

\begin{proof}
  Choose $\bx \in (\Supp \psi)_\Gamma$ and $\bx' \in (\Supp \psi')_\Gamma$.  Then $\bx \cap \bx' = \varnothing$ and, by Proposition~\ref{prop:Weyl-untwisted-ann-tensor}\eqref{prop-item:Weyl-untwisted:tensor}, we have $W(\psi_\bx + \psi'_{\bx'}) \cong W(\psi_\bx) \otimes W(\psi'_{\bx'})$.  Since $\bx \cup \bx' \in (\Supp (\psi + \psi'))_\Gamma$, the proposition follows after restricting to $(\g \otimes A)^\Gamma$-modules.
\end{proof}

%%%%%%%%%%%%%%%%%%%%%%%%%%%%%%%%%%%%%%%%%%%%%%%%%%%%%%%%%%%%%%%%%%%
%
\section{Characterization of local Weyl modules by homological properties}
\label{sec:homological-properties}
%
%%%%%%%%%%%%%%%%%%%%%%%%%%%%%%%%%%%%%%%%%%%%%%%%%%%%%%%%%%%%%%%%%%%

In this section, we show that the local Weyl modules are characterized by homological properties, extending results of \cite{CFK} to the equivariant setting.

For $\lambda \in P$, write $\lambda = \sum_{i \in I} k_i \alpha_i$, $k_i \in \Q$, as a linear combination of simple roots, and define
\[ \ts
  \height \lambda := \sum_{i \in I} k_i.
\]
Recall the usual partial order on $P$ given by
\[
  \lambda \geq \mu \iff \lambda - \mu \in Q^+.
\]
It is clear that
\[
  \lambda > \mu \implies \height \lambda > \height \mu.
\]
Since $\Gamma$ acts on $P^+$ via diagram automorphisms, it preserves the set of positive roots.  Therefore, for $\psi \in \cE^\Gamma$, we have $\sum_{x \in X_\rat} \height \psi_\bx(x) = \sum_{x \in X_\rat} \height \psi_{\bx'}(x)$ for all $\bx,\bx' \in (\Supp \psi)_\Gamma$.

\begin{definition}[Height function on $\cE^\Gamma$]
Define the \emph{height} of $\psi \in \cE^\Gamma$ to be
\[ \ts
  \height \psi = \sum_{x \in X_\rat} \height \psi_\bx(x) \quad \text{for some } \bx \in (\Supp \psi)_\Gamma.
\]
By the above discussion, this definition is independent of the choice of $\bx$.
\end{definition}

For a finite-dimensional $(\g \otimes A)^\Gamma$-module $M$ and $\psi \in \cE^\Gamma$, let $\mult_\psi M$ denote the multiplicity of $\ev^\Gamma_\psi$ in $M$.  In other words, $\mult_\psi M$ is the number of (irreducible) composition factors of $M$ in the isomorphism class $\ev^\Gamma_\psi$.

\begin{definition}[Maximal weight module]
  We call a finite-dimensional $(\g \otimes A)^\Gamma$-module $M$ a \emph{maximal weight module} of \emph{maximal weight} $\psi$ if $\mult_\psi M =1$ and, for all $\varphi \neq \psi$,
  \[
    \mult_\varphi M \neq 0 \implies \height \varphi < \height \psi.
  \]
\end{definition}

\begin{lemma}
  The local Weyl module $W_\Gamma(\psi)$ is a maximal weight module of maximal weight $\psi$.
\end{lemma}

\begin{proof}
  If $\Gamma = \{1\}$, the result follows from the fact that the $\g$-weights of $W(\psi)$ lie in $\wt \psi - Q^+$ by Definition~\ref{def:untwisted-local-Weyl-module}.  Suppose now $\Gamma \ne \{1\}$ and let $\psi \in \mathcal{E}^\Gamma$. Then for any $\bx \in (\Supp \psi)_\Gamma$, we have, by Lemma~\ref{lem:Weyl-simple-correspondences}, $\bU_\bx(W_\Gamma(\psi)) = W(\psi_\bx)$.  By Proposition~\ref{prop:Weyl-untwisted-ann-tensor}\eqref{prop-item:untwisted-Weyl-annihilator}, we have that all constituents of $W(\psi_\bx)$ have support contained in $\bx$.  Thus
  \[
    \mult_\varphi W(\psi_\bx) \ne 0 \implies V(\varphi) \in \cF_\bx.
  \]
  By Theorem~\ref{theo:category-equivalence} and Lemma~\ref{lem:Weyl-simple-correspondences}, we then have
  \[
    \mult_\varphi W_\Gamma(\psi) = \mult_{\varphi_\bx} W(\psi_\bx).
  \]
  Thus, for $\varphi \ne \psi$ (hence $\varphi_\bx \ne \psi_\bx$),
  \begin{align*}
    \mult_\varphi W_\Gamma(\psi) \ne 0 &\implies \mult_{\varphi_\bx} W(\psi_\bx) \ne 0 \\
    &\implies \wt \varphi_\bx < \wt \psi_\bx \\
    &\implies \height \varphi = \height \varphi_\bx < \height \psi_\bx = \height \psi,
  \end{align*}
  where the second implication follows again from the fact that the $\g$-weights of $W(\psi_\bx)$ lie in $\wt \psi_\bx - Q^+$ by Definition~\ref{def:untwisted-local-Weyl-module}.
\end{proof}

Recall that, for $\psi \in \cE$, we have $\wt \psi = \sum_{x \in X_\rat} \psi(x)$.  It is clear that $\wt \psi$ is the maximal $\g$-weight occurring in $V(\psi)$.  We have the following characterization of untwisted local Weyl modules in terms of homological properties.

\begin{proposition}[{\cite[Prop.~8]{CFK}}] \label{prop:untwistedhom}
  Let $M$ be a maximal weight $(\g \otimes A)$-module of maximal weight $\psi$.  Then $M \cong W(\psi)$ if and only if
  \[
    \Hom_\cF(M, V(\varphi)) = 0 \quad \text{and} \quad \Ext^1_\cF(M, V(\varphi)) = 0
  \]
  for all $\varphi \in \cE$ with $\wt(V(\varphi)) \subseteq (\wt \psi - Q^+)\setminus\{\wt \psi\}$.
\end{proposition}

We want to reformulate this theorem and generalize it to the case of equivariant map algebras.

\begin{theorem} \label{theo:twisted-Weyl-module-hom-property}
  Let $M$ be a maximal weight $(\g \otimes A)^\Gamma$-module of maximal weight $\psi$.  Then $M \cong W_\Gamma(\psi)$ if and only if
  \begin{equation} \label{eq:twisted-Weyl-module-hom-property}
    \Hom_{\cF^\Gamma}(M, V_\Gamma(\varphi)) = 0 \text{ and } \Ext^1_{\cF^\Gamma}(M, V_\Gamma(\varphi)) = 0 \ \forall\ \varphi \in \cE^\Gamma \text{ with } \height(\varphi) < \height(\psi).
  \end{equation}
\end{theorem}

\begin{proof}
  We first prove the theorem in the case $\Gamma = \{1\}$, where it is a slightly modified version of Proposition~\ref{prop:untwistedhom}.  In this case $(\g \otimes A)^\Gamma = \g \otimes A$ and $W_\Gamma (\psi) = W(\psi)$.  We first want to show that $W(\psi)$ satisfies
  \[
    \Hom_{\cF}(W(\psi), V(\varphi)) = 0 \quad \text{and} \quad \Ext^1_{\cF}(W(\psi), V(\varphi)) = 0
  \]
  for all $\varphi \in \cE$ with $\height(\varphi) < \height(\psi)$. Since the group $\Gamma$ is trivial, all finite-dimensional $(\g \otimes A)$-modules are also $\g$-modules via the identification of $\g$ with $\g \otimes k \subseteq \g \otimes A$.  Thus we have weight space decompositions as $\g$-modules.

  Let $\lambda = \wt \psi$.  Since $\height \varphi < \height \psi $, we have $\lambda \notin \wt(\varphi) - Q^+$ and so $V(\varphi)_\lambda = 0$. Since $W(\psi)$ is generated by $W(\psi)_\lambda$, this implies $\Hom_{\cF}(W(\psi), V(\varphi)) = 0$.

  Now suppose we have an extension of $(\g \otimes A$)-modules
  \begin{equation} \label{eq:proof-seq-splitting-1}
    0 \to V(\varphi) \to E \to W(\psi) \to 0.
  \end{equation}
  Let $w_\lambda$ be the preimage in $E$ of a maximal weight vector of $W(\psi)$.  Since $\lambda \notin \wt(\varphi) - Q^+$, we have $\dim E_\lambda = 1$, and so $w_\lambda$ is unique up to nonzero scalar multiple.  Also, $(\mathfrak{n}^+ \otimes A) \cdot w_\lambda = 0$ and so we have an exact sequence
  \begin{equation} \label{eq:proof-seq-splitting-2}
    0 \longrightarrow U \longrightarrow U(\g \otimes A) \cdot w_\lambda \longrightarrow W(\psi) \longrightarrow 0
  \end{equation}
  where $U$ is a $\g \otimes A$-module with $U_\lambda = 0$.  Since $\wt(U(\g \otimes A) \cdot w_\lambda) \subseteq \lambda - Q^+$, we have $\wt(U) \subseteq (\lambda - Q^+) \setminus \{\lambda\}$.  Thus Proposition~\ref{prop:untwistedhom} implies that~\eqref{eq:proof-seq-splitting-2} splits, which in turn implies that~\eqref{eq:proof-seq-splitting-1} splits.  Thus $E$ is the trivial extension.  Therefore $\Ext^1_{\cF}(W(\psi), V(\varphi)) = 0$.

  On the other hand, suppose $M$ satisfies~\eqref{eq:twisted-Weyl-module-hom-property}.  We claim that $M$ also satisfies the properties characterizing $W(\psi)$ as given in Proposition~\ref{prop:untwistedhom}.  Let $\varphi \in \cE$ with $\wt(V(\varphi)) \subseteq (\lambda - Q^+) \setminus \{\lambda\}$.  Then $\wt \varphi < \lambda$, hence $\height(\varphi) < \height(\psi)$.  The claim then follows from \eqref{eq:twisted-Weyl-module-hom-property}.   Hence the theorem is true for $\Gamma = \{1\}$.

  Now consider the case of arbitrary $\Gamma$.  Let $\varphi \in \cE^\Gamma$ with $\height(\varphi) < \height(\psi)$. We would first like to show that
  \[
    \Hom_{\cF^\Gamma}(W_\Gamma(\psi), V_\Gamma(\varphi)) = 0 \quad \text{and} \quad \Ext^1_{\cF^\Gamma}(W_\Gamma(\psi), V_\Gamma(\varphi)) = 0.
  \]
  Let $\tau \in \Hom_{\cF^\Gamma}(W_\Gamma(\psi), V_\Gamma(\varphi))$ be nonzero.  Then $\tau$ is surjective since $V_\Gamma(\varphi)$ is irreducible, and so $V_\Gamma(\varphi)$ is isomorphic to a quotient of $W_\Gamma(\psi)$. By Proposition~\ref{prop:fd-annihilator} there exists $\eta$, $\Supp \eta \subseteq X_*$, such that $W_\Gamma(\psi)$ (hence also $V_\Gamma(\varphi)$) is annihilated by $(\g \otimes I_\eta)^\Gamma$. Let $\bx = \Supp \eta$.  Then
  \[
    \Hom_{\cF^\Gamma}(W_\Gamma(\psi), V_\Gamma(\varphi)) \cong \Hom_{\cF^\Gamma_\bx}(W_\Gamma(\psi), V_\Gamma(\varphi)).
  \]
  Now, by Theorem~\ref{theo:category-equivalence} and Lemma~\ref{lem:Weyl-simple-correspondences}, we have
    \[
      \Hom_{\cF^\Gamma_\bx}(W_\Gamma(\psi), V_\Gamma(\varphi)) \cong \Hom_{\cF_\bx}(W(\psi_\bx), V(\varphi_\bx)).
    \]
  Since $\height \varphi_\bx = \height \varphi < \height \psi = \height \psi_\bx$, we conclude $\Hom_{\cF_\bx}(W(\psi_\bx), V(\varphi_\bx)) = 0$ since we know the theorem is true in the untwisted case.  Thus $\tau = 0$ and so $\Hom_{\cF^\Gamma}(W_\Gamma(\psi), V_\Gamma(\varphi)) = 0$.

  Now let
  \begin{equation} \label{eq:proof-splitting-3}
    0 \to V_\Gamma(\varphi) \to E \to W_\Gamma(\psi) \to 0
  \end{equation}
  be an extension of $(\g \otimes A)^\Gamma$-modules with $\height \varphi < \height\psi$. Since $E$ is finite-dimensional, by Proposition~\ref{prop:fd-annihilator} there exists $\eta$, $\Supp \eta \subseteq X_*$, such that $(\g \otimes I_\eta)^\Gamma \cdot E = 0$.  But this implies
  \[
    (\g \otimes I_\eta)^\Gamma \cdot W_\Gamma(\psi) = 0 \quad \text{and} \quad (\g \otimes I_\eta)^\Gamma \cdot V_\Gamma(\varphi) = 0.
  \]
  Thus~\eqref{eq:proof-splitting-3} is an exact sequence in $\cF^\Gamma_\bx$ for $\bx = \Supp \eta$ and hence, by Theorem~\ref{theo:category-equivalence} and Lemma~\ref{lem:Weyl-simple-correspondences},
  \begin{equation} \label{eq:proof-splitting-4}
    0 \to V(\varphi_\bx) \to \bU_\bx E \to W(\psi_\bx) \to 0
  \end{equation}
  is a short exact sequence in $\cF_\bx$.  Since $\height \varphi_\bx = \height \varphi < \height \psi = \height \psi_\bx$, \eqref{eq:proof-splitting-4} splits by the fact that the theorem is true in the untwisted case.  Then Theorem~\ref{theo:category-equivalence} implies that~\eqref{eq:proof-splitting-3} splits.  So $\Ext^1_{\cF^\Gamma}(W_\Gamma(\psi), V_\Gamma(\varphi)) = 0$.

  On the other hand, suppose $M$ satisfies \eqref{eq:twisted-Weyl-module-hom-property}.  We would like to show that $M \cong W_\Gamma(\psi)$.  Fix $\bx \in (\Supp M)_\Gamma$.  Then $M \in \cF_\bx^\Gamma$ and so $\bU_\bx M$ is a module in $\cF_\bx$.  By Theorem~\ref{theo:category-equivalence} and Lemma~\ref{lem:Weyl-simple-correspondences}, it suffices to show that $\bU_\bx M \cong W(\psi_\bx)$. Since $M$ is a maximal weight module of maximal weight $\psi$, we have $\Supp \psi \subseteq \Supp M$, hence $\bx \cap (\Supp \psi) \in (\Supp \psi)_\Gamma$ and $\bU_\bx M$ is a maximal weight module of maximal weight $\psi_\bx$.  In particular, this implies that the $\g$-weight space of $\bU_\bx M$ of weight $\wt \psi_\bx$ is one-dimensional.

  Let $m_\psi$ be a nonzero element of $(\bU_\bx M)_{\wt{\psi_\bx}}$.  We claim that $\bU_\bx M$ is cyclic and generated by $m_\psi$.
  Indeed, if this were not the case, then the submodule generated by $v$, where $v$ is in a $\g$-complement of $U(\g \otimes A) \cdot m_\psi$ would have an irreducible quotient $V(\varphi)$, with $\height \varphi < \height \psi$ and $\Supp \varphi \subseteq \bx$. Then $\bT_\bx (V(\varphi)) = V_\Gamma(\varphi^\Gamma)$ would be an irreducible object of $\mathcal{F}_\bx^\Gamma$.  Again by Theorem~\ref{theo:category-equivalence}, we would have
  \[
    \Hom_{\cF_\bx}(\bU_\bx M, V(\varphi)) \neq 0 \implies \Hom_{\cF_\bx^\Gamma}(M, V_\Gamma(\varphi^\Gamma)) \neq 0,
  \]
  which contradicts~\eqref{eq:twisted-Weyl-module-hom-property} since $\height \varphi^\Gamma = \height \varphi < \height \psi$.  By Proposition~\ref{prop:Weyl-untwisted-quotient}\eqref{prop-item:Weyl-untwisted:quotient}, $\bU_\bx M$ is a quotient of $W(\psi_\bx)$.  It remains to show that it is not a proper quotient.   We have $(\bU_\bx M)_\mu = 0 $ for all $\mu > \lambda$, so $(\mathfrak{n}^+ \otimes A) \cdot m_\psi  = 0$, which implies we have an exact sequence
  \[
    0 \to U \to W(\psi_\bx) \to \bU_\bx M \to 0
  \]
  with $U$ an object of $\cF_\bx$ satisfying $U_\lambda = 0$ and $\wt(U) \subseteq \lambda - Q^+$. Applying $\bT_\bx$, we have
  \begin{equation} \label{eq:proof-ext}
    0 \to \bT_\bx U \to W_\Gamma(\psi) \to M\to 0.
  \end{equation}
  Now applying $\Hom_{\cF_\bx^\Gamma} ( -, V_\Gamma(\varphi))$, for $\varphi \in \cE^\Gamma$ with $\Supp \varphi \subseteq \bx$, to the short exact sequence \eqref{eq:proof-ext}, we obtain the long exact sequence
  \begin{multline*}
    0\to \Hom_{\cF_\bx^\Gamma}(M,V_\Gamma(\varphi)) \to \Hom_{\cF_\bx^\Gamma}(W_\Gamma(\psi),V_\Gamma(\varphi)) \to \Hom_{\cF_\bx^\Gamma}(\bT_\bx U,V_\Gamma(\varphi)) \\
    \to \Ext^1_{\cF_\bx^\Gamma}(M,V_\Gamma(\varphi)) \to \cdots
  \end{multline*}
  By \eqref{eq:twisted-Weyl-module-hom-property}, we have
  \begin{gather*}
    \Hom_{\cF_\bx^\Gamma}(W_\Gamma(\psi), V_\Gamma(\varphi)) = \Hom_{\cF^\Gamma}(W_\Gamma(\psi), V_\Gamma(\varphi)) = 0 \quad \text{and} \\
    \Ext^1_{\cF_\bx^\Gamma}(M,V_\Gamma(\varphi)) = \Ext^1_{\cF^\Gamma}(M,V_\Gamma(\varphi)) = 0
  \end{gather*}
  when $\height \varphi <\height \psi$.  Thus $\Hom_{\cF_\bx^\Gamma}(\bT_\bx U,V_\Gamma(\varphi))=0$, whenever $\height \varphi < \height \psi$.  Since all irreducible subquotients $V_\Gamma(\varphi)$ of $\bT_\bx U$ satisfy $\Supp \varphi \subseteq \bx$ and  $\height \varphi < \psi$, we have $\bT_\bx U = 0$ and hence $U=0$. Thus the theorem follows.
\end{proof}

The following corollary is a twisted version of Proposition~\ref{prop:Weyl-untwisted-quotient}\eqref{prop-item:Weyl-untwisted:quotient}.  Condition~\eqref{eq:twisted-quotient} below should be thought of as a twisted analogue of the condition in Proposition~\ref{prop:Weyl-untwisted-quotient}\eqref{prop-item:Weyl-untwisted:quotient} that $M$ is cyclicly generated by the vector $v$.

\begin{corollary}
  Let $M$ be a maximal weight $(\g \otimes A)^\Gamma$-module of maximal weight $\psi \in \cE^\Gamma$ such that
  \begin{equation} \label{eq:twisted-quotient}
    \Hom_{\cF^\Gamma}(M, V(\varphi)) = 0
  \end{equation}
  for all $\varphi \in \cE^\Gamma$ with $\height \varphi < \height \psi$.  Then $M$ is a quotient of $W_\Gamma(\psi)$.
\end{corollary}

\begin{proof}
  This follows from the proof of Theorem~\ref{theo:twisted-Weyl-module-hom-property}.
\end{proof}

%%%%%%%%%%%%%%%%%%%%%%%%%%%%%%%%%%%%%%%%%%%%%%%%%%%%%%%%%%%%%%%%%%%
% References
%%%%%%%%%%%%%%%%%%%%%%%%%%%%%%%%%%%%%%%%%%%%%%%%%%%%%%%%%%%%%%%%%%%

\bibliographystyle{alpha}
\bibliography{twm-biblist}

\begin{thebibliography}{CFK10}

\bibitem[BN04]{BN04}
Jonathan Beck and Hiraku Nakajima.
\newblock Crystal bases and two-sided cells of quantum affine algebras.
\newblock {\em Duke Math. J.}, 123(2):335--402, 2004.

\bibitem[Bou75]{Bou75}
N.~Bourbaki.
\newblock {\em \'{E}l\'ements de math\'ematique}.
\newblock Hermann, Paris, 1975.
\newblock Fasc. XXXVIII: Groupes et alg{\`e}bres de Lie. Chapitre VII:
  Sous-alg{\`e}bres de Cartan, {\'e}l{\'e}ments r{\'e}guliers. Chapitre VIII:
  Alg{\`e}bres de Lie semi-simples d{\'e}ploy{\'e}es, Actualit{\'e}s
  Scientifiques et Industrielles, No. 1364.

\bibitem[CFK10]{CFK}
Vyjayanthi Chari, Ghislain Fourier, and Tanusree Khandai.
\newblock A categorical approach to {W}eyl modules.
\newblock {\em Transform. Groups}, 15(3):517--549, 2010.

\bibitem[CFS08]{CFS08}
Vyjayanthi Chari, Ghislain Fourier, and Prasad Senesi.
\newblock Weyl modules for the twisted loop algebras.
\newblock {\em J. Algebra}, 319(12):5016--5038, 2008.

\bibitem[Cha86]{C85}
Vyjayanthi Chari.
\newblock Integrable representations of affine {L}ie-algebras.
\newblock {\em Invent. Math.}, 85(2):317--335, 1986.

\bibitem[CL06]{CL06}
Vyjayanthi Chari and Sergei Loktev.
\newblock Weyl, {D}emazure and fusion modules for the current algebra of
  {$\mathfrak{sl}_{r+1}$}.
\newblock {\em Adv. Math.}, 207(2):928--960, 2006.

\bibitem[CM04]{CM04}
Vyjayanthi Chari and Adriano~A. Moura.
\newblock Spectral characters of finite-dimensional representations of affine
  algebras.
\newblock {\em J. Algebra}, 279(2):820--839, 2004.

\bibitem[CP86]{CP86}
Vyjayanthi Chari and Andrew Pressley.
\newblock New unitary representations of loop groups.
\newblock {\em Math. Ann.}, 275(1):87--104, 1986.

\bibitem[CP01]{CP01}
Vyjayanthi Chari and Andrew Pressley.
\newblock Weyl modules for classical and quantum affine algebras.
\newblock {\em Represent. Theory}, 5:191--223 (electronic), 2001.

\bibitem[FL04]{FL04}
B.~Feigin and S.~Loktev.
\newblock Multi-dimensional {W}eyl modules and symmetric functions.
\newblock {\em Comm. Math. Phys.}, 251(3):427--445, 2004.

\bibitem[FL07]{FoL07}
G.~Fourier and P.~Littelmann.
\newblock Weyl modules, {D}emazure modules, {KR}-modules, crystals, fusion
  products and limit constructions.
\newblock {\em Adv. Math.}, 211(2):566--593, 2007.

\bibitem[Her10]{Her10}
David Hernandez.
\newblock Kirillov-{R}eshetikhin conjecture: the general case.
\newblock {\em Int. Math. Res. Not. IMRN}, (1):149--193, 2010.

\bibitem[Kod10]{Kod10}
Ryosuke Kodera.
\newblock Extensions between finite-dimensional simple modules over a
  generalized current {L}ie algebra.
\newblock {\em Transform. Groups}, 15(2):371--388, 2010.

\bibitem[Lau10]{Lau10}
Michael Lau.
\newblock Representations of multiloop algebras.
\newblock {\em Pacific J. Math.}, 245(1):167--184, 2010.

\bibitem[Nak01]{N01}
Hiraku Nakajima.
\newblock Quiver varieties and finite-dimensional representations of quantum
  affine algebras.
\newblock {\em J. Amer. Math. Soc.}, 14(1):145--238, 2001.

\bibitem[Nao]{Na10}
Katsuyuki Naoi.
\newblock Weyl modules, {D}emazure modules and finite crystals for non-simply
  laced type.
\newblock arXiv:1012.5480.

\bibitem[NS]{NS11}
Erhard Neher and Alistair Savage.
\newblock Extensions and block decompositions for finite-dimensional
  representations of equivariant map algebras.
\newblock arXiv:1103.4367.

\bibitem[NSS]{NSS}
Erhard Neher, Alistair Savage, and Prasad Senesi.
\newblock Irreducible finite-dimensional representations of equivariant map
  algebras.
\newblock \emph{Trans. Amer. Math. Soc.} (to appear), arXiv:0906.5189.

\bibitem[VV02]{VV02}
M.~Varagnolo and E.~Vasserot.
\newblock Standard modules of quantum affine algebras.
\newblock {\em Duke Math. J.}, 111(3):509--533, 2002.

\end{thebibliography}

\end{document}